\documentclass[12pt]{amsart} 
\usepackage{amsmath,amssymb}
\usepackage{latexsym}
\usepackage{amssymb} 
\usepackage{latexsym}
\usepackage{delarray} 
\usepackage{mathrsfs}

\usepackage[colorlinks]{hyperref}

\numberwithin{equation}{section}
\theoremstyle{plain}
\newtheorem{thm}{Theorem}[section]
\newtheorem{lem}[thm]{Lemma}
\newtheorem{prop}[thm]{Proposition}

\newtheorem{rem}[thm]{Remark}

\newcommand\R{{\mathbb R}}

\newcommand\Rn{{{\mathbb R}^n}}

\newcommand\pa{\partial}

\def\Re{{\rm Re}}

\title[Gevrey well-posedness of the Kirchhoff equation]
{On the Gevrey well-posedness of the Kirchhoff equation}

\author[Tokio Matsuyama]{Tokio Matsuyama}
\address{
Tokio Matsuyama:
 \endgraf
Department of Mathematics
\endgraf
Chuo University
\endgraf
1-13-27, Kasuga, Bunkyo-ku
\endgraf
Tokyo 112-8551
\endgraf
Japan
\endgraf
{\it E-mail address} {\rm tokio@math.chuo-u.ac.jp}
}
\author[Michael Ruzhansky]{Michael Ruzhansky}
\address{
  Michael Ruzhansky:
  \endgraf
  Department of Mathematics
  \endgraf
  Imperial College London
  \endgraf
  180 Queen's Gate, London SW7 2AZ
  \endgraf
  United Kingdom
  \endgraf
  {\it E-mail address} {\rm m.ruzhansky@imperial.ac.uk}
  }

\thanks{
The second author was supported in parts by the
EPSRC grant EP/K039407/1 and EPSRC Leadership Fellowship EP/G007233/1.
}


\begin{document} 

\begin{abstract}
This paper is devoted to proving the 
almost 
global solvability of the Cauchy problem for the Kirchhoff equation
in the Gevrey space $\gamma^s_{\eta,L^2}$. 
Furthermore, similar results are obtained for the 
initial-boundary value problems in bounded domains and in exterior domains 
with compact boundary. 
\end{abstract}

\maketitle

\section{Introduction} 
\label{sec:1}
G. Kirchhoff proposed the equation 
\[
\pa^2_t u-\left(1+\int_\Omega |\nabla u(t,y)|^2\, dy \right)\Delta u=0 
\quad (t\in \R, \, x\in \Omega)
\]
in his book on mathematical physics in 1876, as a model equation for transversal motion of 
the elastic string, where $\Omega$ is a domain in $\Rn$ 
(see \cite{Kirchhoff}, and for finite dimensional approximation problem, see Nishida 
\cite{Nishida}). 
Since then, it was first in 1940 that Bernstein proved the existence of global in time 
analytic solutions on an interval of real line in his celebrated 
paper \cite{Bernstein}. After him, Arosio and Spagnolo discussed analytic solutions in 
higher spatial dimensions (see \cite{Arosio}), and 
D'Ancona and Spagnolo proved 
analytic well-posedness for the degenerate Kirchhoff equation
(see \cite{Dancona-Invent}, and also Kajitani and Yamaguti \cite{Kajitani-Pisa}). \smallskip

As it is well known, this equation has a Hamiltonian structure,  nevertheless 
it involves a challenging 
problem {\em whether or not, one can prove 
the existence of time global solutions corresponding to data in Gevrey classes, 
$H^\infty$-class or standard Sobolev spaces without smallness condition.}
Up to now, there is no solution to these problems. 

\smallskip
The global existence of quasi-analytic solutions is known, see Ghisi and Gobbino, 
Nishihara, Pohozhaev (\cite{Ghisi,Nishihara,Pohozhaev}). 
Here 
quasi-analytic classes are intermediate ones between the analytic class and the
$C^\infty$-class. Manfrin discussed the time global solutions in Sobolev spaces 
corresponding to non-analytic data having a spectral gap (see \cite{Manfrin-JDE}), 
and a similar result is obtained by Hirosawa (see \cite{Hirosawa}). 
\smallskip

On the other hand, global well-posedness in 
Sobolev space $H^{3/2}$, or $H^2$ with {\em small data} is well established in 
\cite{Callegari,Dancona1,Dancona2,Dancona3,Greenberg,Kajitani,Ext-Matsuyama,
Rend-Matsuyama,Rzymowski,Yamazaki1,Yamazaki2}. There, the classes of small data consist of 
compactly supported functions (see \cite{Greenberg}), or more generally, they are 
characterised by some weight conditions 
(see \cite{Callegari,Dancona1,Dancona2,Dancona3}) or oscillatory integrals 
(see \cite{Kajitani,Heiming,Ext-Matsuyama,Rend-Matsuyama,Racke,Rzymowski,Yamazaki1,Yamazaki2}). 
Recently, the authors 
studied the global well-posedness 
for Kirchhoff systems with small data (see \cite{MR-Liouville}), 
and generalised all the previous results in the framework of small data. 
Here, the class of data in \cite{MR-Liouville} consists of Sobolev space $(H^1)^m$, $m$ being 
the order of system, and 
is characterised by some oscillatory integrals. 
The precise statements of the known results can be found in the survey paper \cite{Isaac}. 
\smallskip

The aim of this paper is to prove 
the almost global existence of solutions 
in Gevrey spaces 
(see Theorem \ref{thm:Gevrey thm}). 
Furthermore, we indicate how to modify the proof to also yield the global existence for the
initial-boundary value problem in
exterior domains, and in 
bounded domains
(see Theorem \ref{thm:Exterior} and Theorem \ref{thm:A}, respectively). 

\smallskip

In this paper we 
consider the Cauchy problem for the 
Kirchhoff equation
\begin{equation}\label{EQ:Kirchhoff}
\left\{
\begin{aligned}
& \pa^2_t u-\left(1+\int_\Rn |\nabla u(t,y)|^2\, dy \right)\Delta u=0, 
& \quad t>0, \quad x\in \Rn,\\
& u(0,x)=u_0(x), \quad \pa_t u(0,x)=u_1(x), &\quad x \in \mathbb{R}^n.
\end{aligned}\right.
\end{equation}
Equation \eqref{EQ:Kirchhoff} has a 
Hamiltonian structure. More precisely, 
let us define the energy{\rm :}
\[
\mathscr{H}(u;t):=\frac12\left\{\|\nabla u(t)\|^2_{L^2}+
\|\pa_t u(t)\|^2_{L^2}\right\}+\frac{1}{4} \|\nabla u(t)\|^4_{L^2}.
\]
Then we have 
\[
\mathscr{H}(u;t)=\mathscr{H}(u;0)
\]
as long as a solution exists (see Lemma 
\ref{lem:Energy}).  
We shall now recall the definition of Gevrey class of $L^2$ type. 
For $s\ge1$, we denote by $\gamma^s_{L^2}=\gamma^s_{L^2}(\Rn)$ the 
Roumieu--Gevrey class of 
order $s$ on $\Rn$:
\[
\gamma^s_{L^2}=\bigcup_{\eta>0}\gamma^s_{\eta,L^2},
\]
where $f$ belong to $\gamma^s_{\eta,L^2}$
if and only if 
\[
\int_\Rn e^{\eta|\xi|^{1/s}}|\widehat{f}(\xi)|^2\, d\xi<\infty,
\]
where $\widehat{f}(\xi)$ stands for the Fourier transform of $f(x)$.
The class $\gamma^s_{L^2}$ is endowed with the 
inductive limit topology.
In particular, if 
$s=1$, then $\gamma^1_{L^2}(\Rn)$ is the class $\mathcal{A}_{L^2}$ 
of the analytic functions on 
$\Rn$. 
We will use the norms
\[
\|f\|_{\gamma^s_{\eta,L^2}}
=\left[\int_\Rn e^{\eta|\xi|^{1/s}}|\widehat{f}(\xi)|^2\, d\xi\right]^{1/2}
\]
and 
\[
\|(f,g)\|_{\gamma^s_{\eta,L^2}\times \gamma^s_{\eta,L^2}}
=\left[\int_\Rn e^{\eta|\xi|^{1/s}}
\left\{|\widehat{f}(\xi)|^2
+|\widehat{g}(\xi)|^2\right\} \, d\xi\right]^{1/2}
\]
for $\eta>0$.

\bigskip

We shall prove here the followings:

\begin{thm} \label{thm:Gevrey thm}
Let $T>0$ and $s>1$. 
Let $M>2$, $R>0$ and denote
$\eta_0(M,R,T)=2s M^2 e^{4M^2}RT^{1+\frac1s}+4M^2.$
If  
the functions
$u_0,u_1\in\gamma^s_{L^{2}}$, for some $\eta>\eta_0(M,R,T)$,
satisfy  conditions
\[
2\mathscr{H}(u;0)<\frac{M^2}{4}-1,
\]
\[
\left\|((-\Delta)^{3/4}u_0,(-\Delta)^{1/4}u_1)
\right\|^2_{\gamma^s_{\eta,L^2}\times \gamma^s_{\eta,L^2}}
\le R,
\]
%
then the Cauchy problem \eqref{EQ:Kirchhoff} admits a unique solution
$u\in C^1([0,T];\gamma_{L^2}^s)$.
\end{thm}

We note that Theorem \ref{thm:Gevrey thm} does not seem to require the
smallness of data. In fact, $M$ and $R$ (measuring the size of the data) are
allowed to be large. However, it follows that $\eta$ (measuring the regularity of the
data) then also have to be big. So, we can informally describe conditions of
Theorem \ref{thm:Gevrey thm} that `the larger the data is the more regular it has to be'.

We can also make the following observation concerning the statement of
Theorem \ref{thm:Gevrey thm}.

\begin{rem}\label{REM:eta0}
The formula for $\eta_{0}(M,R,T)$ in Theorem \ref{thm:Gevrey thm}
comes from condition \eqref{EQ:q} with $s$ and $q$ related by 
\eqref{EQ:s}.
The proof actually yields a more precise conclusion, namely, that the 
solution $u$ from Theorem \ref{thm:Gevrey thm}
satisfies
$$
u \in \bigcap^1_{j=0}\, 
C^j \left([0,T];(-\Delta)^{-(3/4)+(j/2)}\gamma_{\eta^\prime,L^2}^s 
\cap (-\Delta)^{-(1/2)+(j/2)}
\gamma^s_{\eta^\prime,L^2} \right),
$$
with 
\begin{equation}\label{EQ:eta-prime2}
\eta^\prime
=\eta-\eta_0(M,R,T)>0.
\end{equation}
This and the order $\eta^{\prime}$ in \eqref{EQ:eta-prime2} can be found from
\eqref{EQ:Gevrey data} and \eqref{EQ:eta-prime} with 
$s$ and $q$ related by 
\eqref{EQ:s}.
\end{rem}

This paper is organised as follows; in \S \ref{sec:2} energy 
estimates for linear equations with time-dependent coefficients will be derived, 
and these estimates will be applied to get a priori estimates. 
Sections \ref{sec:3} will be devoted to proof of 
Theorems \ref{thm:Gevrey thm}.
In \S \ref{sec:4} some results on global well-posedness 
for the initial-boundary value problems will be discussed. \\

\noindent
{\bf Acknowledgements.} The authors would like to express their sincere 
gratitude to Professors Taeko Yamazaki, Kenji Nishihara, Makoto Nakamura 
and Doctor Tsukasa Iwabuchi for fruitful discussions. The authors would like to 
thank also to Professors Kiyoshi Mochizuki, Hiroshi Uesaka and Masaru Yamaguchi 
for giving them many useful advices.


\section{Energy estimates for linear equation}
\label{sec:2}
In this section we shall derive energy 
estimates for solutions
of the linear Cauchy problem with 
time-dependent coefficients.
These estimates will be fundamental tools 
in the proof of the theorems.  

\smallskip

Let us consider the linear Cauchy problem
\begin{equation} \label{EQ:Linear}
\left\{
\begin{aligned}
& \pa^2_t u-c(t)^2\Delta u=0, & \quad 
t\in (0,T), \quad x\in \Rn,\\
& u(0,x)=u_0(x), \quad \pa_t u(0,x)=u_1(x), 
&\quad x \in \mathbb{R}^n.
\end{aligned}\right.
\end{equation} 
The assumptions for the following estimates are related with 
Theorem 2 from Colombini, Del Santo and Kinoshita \cite{Colombini}.
However, here we need more precise conclusions on the behaviour of constants.
\begin{prop}  
\label{prop:G-Colombini}
Let $\sigma\ge1$ and $1\le s<q/(q-1)$ 
for some $q>1$. 
Assume that $c=c(t)\in 
\mathrm{Lip}_{\mathrm{loc}}([0,T))$ satisfies 
\begin{equation}\label{EQ:G-hyp1}
m_0\le c(t) \le M, \quad t\in [0,T],
\end{equation}
\begin{equation}\label{EQ:G-hyp2}
\left|c^\prime(t)\right|\le 
\frac{K}{(T-t)^q}, \quad a.e.\, t\in [0,T),
\end{equation}
for some $0<m_0<M$ and $K>0$.
If $((-\Delta)^{\sigma/2}u_0,
(-\Delta)^{(\sigma-1)/2}u_1)
\in \gamma^s_{\eta,L^2}
\times \gamma^s_{\eta,L^2}$ for some 
$\eta$ satisfying
\begin{equation}\label{EQ:eta}
\eta>\frac{2Km^{-1}_0}{q-1}+4M^2m^{-1}_0,
\end{equation}
then 
the Cauchy problem \eqref{EQ:Linear} admits 
a unique solution 
$$u\in 
\displaystyle{\bigcap_{j=0}^1}\, 
C^j\left([0,T];
(-\Delta)^{-(\sigma-j)/2}\gamma^s_{\eta^\prime,L^2}\right)$$
such that
\begin{align}\label{EQ:G-interval}
& m_0^2 \|(-\Delta)^{\sigma/2}u(t)\|^2_{\gamma^s_{\eta^\prime,L^2}}
+\|(-\Delta)^{(\sigma-1)/2}\pa_t u(t)\|^2_{\gamma^s_{\eta^\prime,L^2}}\\
\le& \max\{M^2,1\}e^{4M^2m^{-1}_0\max\{1,T^{1-(qs-s)}\}}
\|((-\Delta)^{\sigma/2}u_0,(-\Delta)^{(\sigma-1)/2}u_1)\|^2_{\gamma^s_{\eta,L^2}
\times \gamma^s_{\eta,L^2}}
\nonumber
\end{align}
for $t\in [0,T]$, where 
\[
\eta^\prime
=\eta-\left(\frac{2Km^{-1}_0}{q-1}+4M^2m^{-1}_0\right)>0.
\]
\end{prop}
\begin{proof}
Let $v=v(t,\xi)$ be a solution of the Cauchy problem 
\[
\left\{
\begin{aligned}
& \pa^2_t v+c(t)^2|\xi|^2v=0, & \quad t\in (0,T),\\
& v(0,\xi)=\widehat{u}_0(\xi), \quad \pa_t v(0,\xi)=\widehat{u}_1(\xi). & {} 
\end{aligned}
\right.
\]
We define 
\[
c_*(t,\xi)=
\begin{cases}
c(T) & \quad \text{if $T|\xi|^{1/(qs-s)}\le1$,}\\
c(t) & \quad \text{if $T|\xi|^{1/(qs-s)}>1$ and $0\le t\le T-|\xi|^{-1/(qs-s)}$,}\\
c(T-|\xi|^{-1/(qs-s)}) & \quad \text{if $T|\xi|^{1/(qs-s)}>1$ and $T-|\xi|^{-1/(qs-s)}<t\le T$,}
\end{cases}
\]
and 
\[
\alpha(t,\xi)=2Mm^{-1}_0|c_*(t,\xi)-c(t)||\xi|
+\dfrac{2|c^\prime_*(t,\xi)|}{c_*(t,\xi)}.
\]
We adopt an energy for $v$ as 
\[
E(t,\xi)=\left\{
|v^\prime(t)|^2+c_*(t,\xi)^2 |\xi|^2|v(t)|^2
\right\}k(t,\xi),
\]
where 
\[
k(t,\xi)=|\xi|^{2(\sigma-1)}\exp\left(-\int^t_0 \alpha(\tau,\xi)\, d\tau+\eta|\xi|^{1/s}\right)
\]
and $\eta$ is as in \eqref{EQ:eta}.
We put
\[
\mathcal{E}(t)=\int_\Rn E(t,\xi)\, d\xi.
\]
Hereafter we concentrate on estimating the integral of $\alpha(t,\xi)$.
When $T|\xi|^{1/(qs-s)}\le1$, we can estimate, by using assumption
\eqref{EQ:G-hyp1}
on $c(t)$, 
\begin{align}\label{EQ:IMP1}
\int^t_0\alpha(\tau,\xi)\, d\tau
=& \int^{T}_0 2Mm^{-1}_0|c_*(\tau,\xi)-c(\tau)||\xi|\, d\tau\\
\le& 4M^2m^{-1}_0T|\xi|
\nonumber\\ 
\le& 4M^2 m_0^{-1}T^{1-(qs-s)},
\nonumber
\end{align}
and when $T|\xi|^{1/(qs-s)}>1$, we can estimate
\begin{align}\nonumber
\int^t_0\alpha(\tau,\xi)\, d\tau
\le& \int^{T-|\xi|^{-1/(qs-s)}}_0 \frac{2|c^\prime(\tau)|}{c(\tau)}\, d\tau
+\int^T_{T-|\xi|^{-1/(qs-s)}}2Mm^{-1}_0|c_*(\tau,\xi)-c(\tau)||\xi|\, d\tau\\
\le& \int^{T-|\xi|^{-1/(qs-s)}}_0 \frac{2Km^{-1}_0}{(T-\tau)^q}\, d\tau
+4M^2m^{-1}_0|\xi|^{1-1/(qs-s)}
\label{EQ:IMP2}\\
\le& \frac{2Km^{-1}_0|\xi|^{1/s}}{q-1}+4M^2m^{-1}_0|\xi|^{1-1/(qs-s)}.
\nonumber
\end{align}
Since $1-1/(qs-s)<1/s$, it follows that
\[
|\xi|^{1-1/(qs-s)}\le 1+|\xi|^{1/s}.
\]
Consequently, we get
\[
k(t,\xi)\ge e^{-4M^2m^{-1}_0\max\{1,T^{1-(qs-s)}\}}
|\xi|^{2(\sigma-1)}
e^{\left(\eta-\frac{2Km^{-1}_0}{q-1}-4M^2m^{-1}_0 \right)|\xi|^{1/s}},
\]
and hence,
\begin{multline}\label{EQ:G-Energy1}
\mathcal{E}(t)\ge 
e^{-4M^2m^{-1}_0\max\{1,T^{1-(qs-s)}\}} \times \\
\int_\Rn e^{\left(\eta-\frac{2Km^{-1}_0}{q-1}-4M^2m^{-1}_0\right)|\xi|^{1/s}}
|\xi|^{2(\sigma-1)}\{m^2_0 |\xi|^2|v(t)|^2+|v^\prime(t)|^2\}\, d\xi.
\end{multline}

We compute the derivative of $E(t,\xi)$:
\begin{align*}
& E^\prime(t,\xi)=\\
& \left[2\mathrm{Re} \big\{v^{\prime\prime}(t)\overline{v^\prime(t)}\big\}
+2c_*(t,\xi) c_*^\prime(t,\xi)|\xi|^2|v(t)|^2
+2c_*(t,\xi)^2 |\xi|^2\mathrm{Re} \big\{v^\prime(t)\overline{v(t)}\big\}
\right]k(t,\xi)\\
{} & -\{c_*(t,\xi)^2|\xi|^2 |v(t)|^2+|v^\prime(t)|^2\} \alpha(t,\xi)k(t,\xi)\\
=& \left[2\{c_*(t,\xi)^2-c(t)^2\}|\xi|^2\mathrm{Re}\big\{v^\prime(t)\overline{v(t)}\big\}
+2c_*(t,\xi)c_*^\prime(t,\xi)|\xi|^2|v(t)|^2\right]k(t,\xi)\\
& -\alpha(t,\xi)E(t,\xi).
\end{align*}
Then we can estimate the right hand side as
\begin{align*}
{} & \left[\frac{2|c_*(t,\xi)^2-c(t)^2||\xi|}{c_*(t,\xi)}
\left|v^\prime(t)\right| \cdot c_*(t,\xi)|\xi| \left|v(t)\right|
+2\frac{|c_*^\prime(t,\xi)|}{c_*(t,\xi)} c_*(t,\xi)^2 |\xi|^2|v(t)|^2\right]k(t,\xi)\\
{} &-\alpha(t,\xi)E(t,\xi)\\
\le& 2Mm^{-1}_0|c_*(t,\xi)-c(t)| |\xi| E(t,\xi)
+\frac{2|c^\prime_*(t,\xi)|}{c_*(t,\xi)}E(t,\xi)-\alpha(t,\xi)E(t,\xi)\\
=&0,
\end{align*}
which implies that $E^\prime(t,\xi)\le0$ for $t\in(0,T)$, and we find that
\[
\mathcal{E}(t)\le \mathcal{E}(0).
\]
Thus the required estimate \eqref{EQ:G-interval} follows from this estimate and 
\eqref{EQ:G-Energy1}. 
The proof of Proposition \ref{prop:G-Colombini} is now finished. 
\end{proof}


\section{Proof of Theorem \ref{thm:Gevrey thm}}
\label{sec:3}
We denote by 
\[
H^\sigma=H^\sigma(\Rn)=(1-\Delta)^{-\sigma/2}L^2(\Rn)
\] 
for $\sigma\in \R$ the standard Sobolev spaces,
and their homogeneous version is 
\[
\dot{H}^\sigma=\dot{H}^\sigma(\Rn)=(-\Delta)^{-\sigma/2}L^2(\Rn).
\]
Kirchhoff equation has a Hamiltonian structure. Namely, we have:
\begin{lem} \label{lem:Energy}
Let $u\in \displaystyle{\bigcap_{j=0}^1}\, C^j([0,T_u);H^{(3/2)-j})$ be the solution of \eqref{EQ:Kirchhoff}.
Then we have 
\[
\mathscr{H}(u;t)=\mathscr{H}(u;0), 
\quad \forall t\in [0,T_u),
\]
where we recall that 
\[
\mathscr{H}(u;t)=\frac12\left\{\|\nabla u(t)\|^2_{L^2}+
\|\pa_t u(t)\|^2_{L^2}\right\}+\frac{1}{4} \|\nabla u(t)\|^4_{L^2}.
\]
\end{lem}
\begin{proof}
The proof of Lemma \ref{lem:Energy} is elementary. 
Multiplying equation \eqref{EQ:Kirchhoff} by 
$\pa_t u$ and integrating, we get
\[
\frac{d}{dt}\mathscr{H}(u;t)=0,
\]
as desired.
\end{proof}

Now we consider the {\em linear} Cauchy problem  in the strip 
$(0,T_u(v_0,v_1))\times \Rn$: 
\begin{equation}\label{EQ:G-V-Linear}
\pa^2_t v-c(t)^2\Delta v=0, 
\quad t\in (0,T), \quad x\in \Rn,
\end{equation}
with initial condition 
\begin{equation}\label{EQ:G-V-initial}
v(0,x)=v_0(x), \quad 
\pa_t v(0,x)=v_1(x).
\end{equation}
Here $c(t)$ belongs to a class $\mathscr{K}$ defined as follows:\\

\noindent 
{\bf Class $\mathscr{K}(T)$.} {\em 
Let $T>0$. 
Given constants $q>1$, $M>1$ and $K_0>0$, 
we say that $c(t)$ belongs to 
$\mathscr{K}(T)=\mathscr{K}(q,M,K_0,T)$ if $c(t)$ belongs to 
$\mathrm{Lip}_{\mathrm{loc}}([0,T))$ and satisfies
\[
1\le c(t)\le M, \quad t \in [0,T],
\]
\[
\left| c^\prime(t) \right|\le \frac{K_0}{(T-t)^q}, 
\quad a.e.\, t\in [0,T).
\] 
} 

\smallskip

By the energy estimate \eqref{EQ:G-interval} from 
Proposition \ref{prop:G-Colombini}, there exists a real $\eta>0$
such that if $(v_0,v_1)\in 
(-\Delta)^{-3/4}\gamma^s_{\eta,L^2}\times 
(-\Delta)^{-1/4}\gamma^s_{\eta,L^2}$, then 
the Cauchy problem \eqref{EQ:G-V-Linear}--\eqref{EQ:G-V-initial}
admits a unique solution $v$ satisfying 
\begin{equation}\label{EQ:Gevrey data}
v \in \bigcap^1_{j=0}\, 
C^j \left([0,T];(-\Delta)^{-(3/4)+(j/2)}\gamma_{\eta^\prime,L^2}^s 
\cap (-\Delta)^{-(1/2)+(j/2)}
\gamma^s_{\eta^\prime,L^2} \right),
\end{equation}
provided that $1\le s <q/(q-1)$ and $q>1$, 
where $\eta^\prime>0$ is the real number satisfying   
\begin{equation}\label{EQ:eta-prime}
\eta^\prime
=\eta-\left(\frac{2K_0}{q-1}+4M^2\right)>0.
\end{equation}
If we define the function
\[
\tilde{c}(t):=\sqrt{1+\int_\Rn |\nabla v(t,x)|^2\, dx},
\]
this defines the mapping 
$$\Theta:c(t)\mapsto \tilde{c}(t).
$$
We will show
the compactness of $\mathcal{K}(T)$ in $L^\infty_{\mathrm{loc}}([0,T))$ 
and the continuity of $\Theta$. The convexity of $\mathcal{K}(T)$ is clear. 
If we show that $\Theta$ maps 
$\mathcal{K}(T)$ into itself, the Schauder--Tychonoff fixed point theorem allows us to conclude 
the proof. \\

We shall prove here the following:

\begin{prop}\label{prop:Ge-Chiave}
Let $M>2$, $T>0$ and $R>0$. Let 
$1<q<2$ and $s>1$ be such that 
\begin{equation}\label{EQ:s}
s=\frac{1}{q-1}. 
\end{equation}
Let $\eta>0$ be 
such that
\begin{equation}\label{EQ:q}
\eta>\frac{2M^2 e^{4M^2}RT^q}{q-1}+4M^2.
\end{equation}
If $(v_0,v_1)\in 
(-\Delta)^{-3/4}\gamma^s_{\eta,L^2}\times 
(-\Delta)^{-1/4}\gamma^s_{\eta,L^2}$ 
satisfy 
\begin{equation}\label{EQ:2-Ge-data}
2\mathscr{H}(v;0)\le \frac{M^2}{4}-1,
\end{equation}
\begin{equation}\label{EQ:R-data}
\left\|((-\Delta)^{3/4}v_0,(-\Delta)^{1/4}v_1)
\right\|^2_{\gamma^s_{\eta,L^2}\times \gamma^s_{\eta,L^2}}
\le R,
\end{equation}
then, setting
\begin{equation}\label{EQ:K0}
K_0=M^2 e^{4M^2}RT^q,
\end{equation}
we have the following statement{\rm :}
For any $c(t)\in \mathscr{K}(T)$, let $v$ 
be a solution to the Cauchy problem 
\eqref{EQ:G-V-Linear}--\eqref{EQ:G-V-initial} satisfying \eqref{EQ:Gevrey data}.
Then 
\begin{equation} \label{EQ:Ge-Core1}
1\le \tilde{c}(t) \le M, \quad t\in[0,T],
\end{equation}
\begin{equation} \label{EQ:Ge-Core2}
\left|\tilde{c}^\prime(t)\right| \le \frac{K_0}{(T-t)^q}, \quad t\in[0,T).
\end{equation}
\end{prop}
\begin{proof}
First, we prove \eqref{EQ:Ge-Core2}. To this end, we have only to show that
\begin{equation}\label{EQ:Ge-another}
|\tilde{c}^\prime(t)|\le \frac{K_0}{T^q}
\end{equation}
for $t\in[0,T]$, since the right hand side of \eqref{EQ:Ge-another}
is bounded by $K_0/(T-t)^q$ for $t\in[0,T)$.
One can readily see that 
\[
2\tilde{c}(t)\tilde{c}^\prime(t)=2\Re \left((-\Delta)^{3/4}v(t),(-\Delta)^{1/4}\pa_t 
v(t)\right)_{L^2},
\]
and hence, 
we have
\begin{align}\label{EQ:Gev-est}
|\tilde{c}^\prime(t)|
\le& \left\|v(t)\right\|_{\dot{H}^{3/2}} 
\left\|\pa_t v(t)\right\|_{\dot{H}^{1/2}}\\
\le& \left\|(-\Delta)^{3/4}v(t)
\right\|_{\gamma^s_{\eta^\prime,L^2}}
\left\|(-\Delta)^{1/4}\pa_t v(t)
\right\|_{\gamma^s_{\eta^\prime,L^2}}
\nonumber
\end{align}
for any $\eta^\prime>0$, 
since $\tilde{c}(t)\ge 1$. 
Then, by the definition 
\eqref{EQ:K0} of $K_0$, the constant $\eta$
satisfies the following inequality:
\[
\eta>\frac{2K_0}{q-1}+4M^2.
\]
Hence, if 
$\eta^\prime$ is chosen as in 
\eqref{EQ:eta-prime},
then, applying the energy estimate 
\eqref{EQ:G-interval} from 
Proposition \ref{prop:G-Colombini} to the 
right hand side of 
\eqref{EQ:Gev-est}, we can write
\begin{equation}\label{EQ:preparation}
|\tilde{c}^\prime(t)|
\le M^2
e^{4M^2\max\{1,T^{1-(qs-s)}\}}
\left\|\left((-\Delta)^{3/4}v_0,(-\Delta)^{1/4}v_1\right)
\right\|^2_{\gamma^s_{\eta,L^2}\times \gamma^s_{\eta,L^2}}
\end{equation}
for $t\in[0,T]$. Since $1-(qs-s)=0$ by assumption \eqref{EQ:s},
it follows that 
\begin{equation}\label{EQ:EXP}
e^{4M^2\max\{1,T^{1-(qs-s)}\}}
=e^{4M^2}.
\end{equation}
Hence, recalling the definition 
\eqref{EQ:K0} of $K_0$ and using 
\eqref{EQ:R-data}, we conclude from 
\eqref{EQ:preparation}--\eqref{EQ:EXP} that
\begin{multline}\label{EQ:est-c}
|\tilde{c}^\prime(t)|
\le M^2
e^{4M^2}
\left\|((-\Delta)^{3/4}v_0,(-\Delta)^{1/4}v_1)
\right\|^2_{\gamma^s_{\eta,L^2}\times \gamma^s_{\eta,L^2}}\\
\le M^2 e^{4M^2}\cdot RT^q \cdot \frac{1}{T^q}
= \frac{K_0}{T^q}
\end{multline}
for $t\in[0,T]$.
Thus we get the required estimate \eqref{EQ:Ge-another}.

Finally we prove \eqref{EQ:Ge-Core1}. In 
this case, we will not use the energy 
estimate \eqref{EQ:G-interval} from 
Proposition \ref{prop:G-Colombini}.
Our assumption \eqref{EQ:2-Ge-data} implies that 
\[
1\le \tilde{c}(0)\le \sqrt{1+2\mathscr{H}(v;0)}\le 
\frac{M}{2}.
\] 
Since $\tilde{c}(t)$ is continuous, 
there exists a time $t_1<T$
such that
\[
1\le \tilde{c}(t)\le M
\]
for $0\le t\le t_1$.
Fixing data $(v_0,v_1)$ satisfying 
\eqref{EQ:2-Ge-data}--\eqref{EQ:R-data},
we can show that the class $\mathcal{K}(t_1,K_0)$ 
is the convex and compact subset of
the Banach space $L^\infty([0,t_1])$, 
and resorting to \eqref{EQ:Ge-Core2}, 
we can also prove that $\Theta$ is continuous 
from $\mathcal{K}(t_1, K_0)$ into itself. 
This argument will be also done in the 
whole interval $[0,T]$ in the last step, where we give its details. 
Then Schauder's fixed point theorem allows us 
to conclude that $\Theta$ has a fixed point 
in $\mathcal{K}(t_1,K_0)$:
\[
c(t)=\Theta(c(t))=\tilde{c}(t)
\] 
for $0\le t\le t_1$. This means that 
solution $v(t,x)$ to 
the linear Cauchy problem 
\eqref{EQ:G-V-Linear}--\eqref{EQ:G-V-initial} 
is also a solution to the nonlinear 
Cauchy problem \eqref{EQ:Kirchhoff} with 
data $(v_0,v_1)$ on $[0,t_1]$. Hence it 
follows from Lemma \ref{lem:Energy} 
and assumption \eqref{EQ:2-Ge-data} that 
\[
2\mathscr{H}(v;t)=2\mathscr{H}(v;0)\le\frac{M^2}{4}-1, 
\quad t\in[0,t_1],
\]
and as a result, we deduce that
\[
1\le \tilde{c}(t)\le \sqrt{1+2\mathscr{H}(v;t)}\le 
\frac{M}{2}
\]
for $0\le t\le t_1$.
Therefore, by the continuity of 
$\tilde{c}(t)$, there exists a 
time $t_2\in(t_1,T)$
such that
\[
1\le \tilde{c}(t)\le M
\]
for $0\le t\le t_2$. Hence, we can develop 
the previous fixed point 
argument; the solution $v(t,x)$ to 
the linear Cauchy problem 
\eqref{EQ:G-V-Linear}--\eqref{EQ:G-V-initial} 
is also a solution to the nonlinear 
Cauchy problem \eqref{EQ:Kirchhoff} with data 
$(v_0,v_1)$ on $[0,t_2]$ satisfying 
\[
2\mathscr{H}(v;t)=2\mathscr{H}(v;0)\le \frac{M^2}{4}-1, 
\quad t\in[0,t_2],
\]
where we have used assumption 
\eqref{EQ:2-Ge-data} in the last step.
Now, we define a time $t_*$ by the maximal 
time such that 
\[
1\le \tilde{c}(t)\le M
\]
for $0\le t\le t_*$. Suppose that 
$t_*<T$. Then, after employing 
the fixed point argument on the interval 
$[0,t_*]$, 
we deduce from Lemma \ref{lem:Energy} and 
assumption \eqref{EQ:2-Ge-data} 
that 
\[
2\mathscr{H}(v;t_*)=2\mathscr{H}(v;0)\le \frac{M^2}{4}-1, 
\]
and hence, we get 
\[
1\le \tilde{c}(t_*)\leq \frac{M}{2}.
\]
Therefore, the fixed point argument will be 
also applicable, and $v(t,x)$ coincides with 
the solution to \eqref{EQ:Kirchhoff} with 
data $(v_0,v_1)$ on some interval 
$[0,t_{**}]$ strictly containing $[0,t_*]$. 
This implies that $\tilde{c}(t)$ is bounded by 
$M$ on $[0,t_{**}]$.
But this contradicts the maximality of 
$t_*$. Thus we must have
the required estimate \eqref{EQ:Ge-Core1}.
The proof of Proposition \ref{prop:Ge-Chiave} is now complete.
\end{proof}

Based on the previous proposition, we prove our theorem. 

\begin{proof}[Completion of proof of Theorem \ref{thm:Gevrey thm}]
Hereafter, we write 
\[
\mathcal{K}=\mathcal{K}(T).
\]
Let $c(t)\in \mathscr{K}$, 
we fix the data 
\[
(v_0,v_1)\in (-\Delta)^{-3/4}\gamma^s_{\eta,L^2}\times 
(-\Delta)^{-1/4}\gamma^s_{\eta,L^2}
\]
satisfying 
\eqref{EQ:2-Ge-data}--\eqref{EQ:R-data}. 
Then it follows from Proposition \ref{prop:Ge-Chiave} that the mapping
\[
\Theta: c(t)\mapsto \tilde{c}(t)
\]
maps from $\mathscr{K}$ into itself. 
Now $\mathscr{K}$ may be regarded as 
the convex subset of the 
Fr\'echet space $L^\infty_{\mathrm{loc}}([0,T))$, 
and we endow 
$\mathscr{K}$ with the induced topology. We shall prove 
the compactness of $\mathscr{K}$ and continuity 
of the mapping $\Theta$. Then the Schauder-Tychonoff 
theorem allows us to conclude 
the proof. 

\smallskip

{\em Compactness of $\mathscr{K}$.} 
We show that $\mathscr{K}$ is uniformly bounded and 
equi-continuous on every compact interval of 
$[0,T)$. Let $\{ c_k(t) \}_{k=1}^\infty$ be a 
sequence in $\mathscr{K}$ such that
\begin{equation}\label{EQ:c-bdd}
1\le c_k(t) \le M, \quad t\in [0,T],
\end{equation}
\begin{equation}\label{EQ:small-S-Calculus}
|c^\prime_k(t)|\le \frac{K_0}{(T-t)^q},  \quad a.e.\, t\in [0,T).
\end{equation}
Observing
\[
c_k(t)-c_k(t^\prime)
=\int^{t}_{t^\prime} c_k^\prime(\tau) \, d \tau, 
\]
we obtain from \eqref{EQ:small-S-Calculus} that
\[
|c_k(t)-c_k(t^\prime)|\le \frac{K_0}{q-1}\left\{\frac{1}{(T-t^\prime)^{q-1}}
-\frac{1}{(T-t)^{q-1}}\right\}
\]
for $0\le t^\prime<t<T$. 
Since $1/(T-t)^{q-1}$ is uniformly continuous 
on every compact interval of $[0,T)$, the sequence 
$\{ c_k(t) \}_{k=1}^\infty$ is equi-continuous on that interval. 
Thus $\mathscr{K}$ is relatively compact in 
$L^\infty_{\mathrm{loc}}([0,T))$, and hence, one can deduce from 
the Arzel\`a-Ascoli theorem that every sequence $\{ c_k(t) \}_{k=1}^\infty$ 
in $\mathscr{K}$ has a subsequence, denoted by the same, 
converging to some $c(\cdot)\in 
L^\infty_{\mathrm{loc}}([0,T))$:
\begin{equation}\label{EQ:Tu}
\left\{
\begin{aligned}
& c_k(t) \underset{(k\to\infty)}{\to} c(t) 
\quad 
\text{in $L^\infty_{\mathrm{loc}}([0,T))$;} \\
& \text{$1\le c(t)\le M$ \quad for every compact interval in 
$[0,T)$;} \\
& \text{$|c(t)-c(t^\prime)|\le \frac{K_0}{q-1}\left\{\frac{1}{(T-t^\prime)^{q-1}}
-\frac{1}{(T-t)^{q-1}}\right\}$, 
\quad $0\le t^\prime<t<T.$}
\end{aligned}\right.
\end{equation}
The last statement of \eqref{EQ:Tu} 
implies that $c(t)$ is in $\mathrm{Lip}_{\mathrm{loc}}([0,T))$,
since the function 
$(T-t)^{-(q-1)}$ is in $\mathrm{Lip}_{\mathrm{loc}}([0,T))$.
Furthermore, $c(t)$ must be 
bounded by $M$ even at $t=T$:
\begin{equation}\label{EQ:bdd-M}
1\le c(t)\le M, \quad t\in[0,T].
\end{equation}
Indeed, if 
\[
\varlimsup_{t\nearrow T}c(t)>M,
\]
there exists a sequence $\{t_j\}$ such that
\[
t_j \nearrow T \quad \text{and} \quad c(t_j)>M, \quad (j=1,2,\ldots).
\]
Going back to \eqref{EQ:c-bdd}, and resorting to the first statement 
of \eqref{EQ:Tu}, we have
\[
c(t_j)=\lim_{k\to\infty}c_k(t_j)\le M, \quad (\forall\, j)
\]
which leads to a contradiction.
Thus we conclude that $c(t)$ satisfies \eqref{EQ:bdd-M} and
\[
c(\cdot)\in \mathrm{Lip}_{\mathrm{loc}}([0,T)),
\]
and the derivative $c^\prime(t)$
exists a.e. $t\in [0,T)$. 
Now, for the derivative $c^\prime(t)$, 
if we prove that
\begin{equation}\label{EQ:small-S-limsup}
|c^\prime(t)|\le \frac{K_0}{(T-t)^q}, \quad a.e.\, t\in [0,T),
\end{equation} 
then $c(t) \in \mathscr{K}$, which proves the compactness of $\mathscr{K}$. 
We prove \eqref{EQ:small-S-limsup}.
Let $t_0\in (0,T)$ be an arbitrary point where $c(t)$ is differentiable. 
Since we have, by using \eqref{EQ:small-S-Calculus}, 
\begin{align*}
& \left| \frac{1}{2h}\left\{c_k(t_0+h)-c_k(t_0-h) \right\}\right|
= \left|\frac{1}{2h} \int_{t_0-h}^{t_0+h}c^\prime_k(t)\, dt \right| \\
\le & \frac{K_0}{2h(q-1)} \left\{\frac{1}{(T-(t_0-h))^{q-1}}
-\frac{1}{(T-(t_0+h))^{q-1}}\right\}
\end{align*}
for $h>0$, we can take the limit in this equation 
with respect to $k$, so that 
\[
\left|\frac{1}{2h}\left\{c(t_0+h)-c(t_0-h)\right\}\right|
\le \frac{K_0}{2h(q-1)} \left\{\frac{1}{(T-(t_0-h))^{q-1}}
-\frac{1}{(T-(t_0+h))^{q-1}}\right\}.
\]
Then, letting $h\to+0$, we conclude that 
\[
|c^\prime(t_0)|
\le \frac{K_0}{(T-t_0)^q}. 
\] 
Since $t_0$ is arbitrary, we get \eqref{EQ:small-S-limsup}. 

\smallskip

{\em Continuity of $\Theta$ on $\mathscr{K}$.} 
Let us take a sequence $\{ c_k(t) \}$ in $\mathscr{K}$ such that 
\[
\text{$c_k(\cdot) \to c(\cdot) \in \mathscr{K}$ \quad in 
$L^{\infty}_{\mathrm{loc}}([0,T))$ \quad $(k \to \infty)$,} 
\]
and let $v_k(t,x)$ and $v(t,x)$ be corresponding solutions 
to the linear Cauchy problem \eqref{EQ:G-V-Linear}--\eqref{EQ:G-V-initial} 
with coefficients $c_k(t)$ and $c(t)$, respectively, with fixed data 
$(v_0,v_1)$. 
Then it is sufficient to  prove that the images 
$\tilde{c}_k(t) :=\Theta (c_k(t))$ and $\tilde{c}(t):=\Theta(c(t))$ satisfy 
\begin{equation} \label{EQ:small-S-convergence}
\text{$\tilde{c}_k(\cdot) \to \tilde{c}(\cdot)$ \quad in 
$L^{\infty}_{\mathrm{loc}}([0,T))$ \quad $(k \to \infty)$.} 
\end{equation} 
The functions 
$w_k:=v_k-v$, $k=1,2,\ldots$, solve the following Cauchy problem: 
\[ \begin{cases} 
\partial^2_t w_k-c(t)^2\Delta w_k 
=\left\{ c_k(t)^2-c(t)^2 \right\} \Delta v_k, 
\quad (t,x) \in (0,T) \times \Rn,\\
w_k(0,x)=0, \quad \partial_t w_k(0,x)=0, \quad x \in \Rn.
\end{cases} 
\]
If we differentiate the energy $\mathscr{E}(w_k(t))$ for $w_k$ with respect to $t$, where 
\[
\mathscr{E}(w_k(t))=\Vert \pa_t w_k(t) \Vert^{2}_{L^2}
+c(t)^2 \| \nabla w_k(t) \|^{2}_{L^2},
\]
we get 
\begin{align} \label{EQ:small-S-diff} 
\mathscr{E}^{\prime}(w_k(t))=& -2\left\{c_k(t)^2-c(t)^2 \right\} 
\Re \left(\Delta v_k(t),\pa_t w_k(t) \right)_{L^2}\\ 
& +2c(t) c^{\prime}(t) 
\left\Vert \nabla w_k(t) \right\Vert^{2}_{L^2}
\nonumber\\ 
\le& 2\left\vert c_k(t)^2-c(t)^2 \right\vert 
\Vert v_k(t) \|_{\dot{H}^{3/2}} \Vert \pa_t w_k(t) 
\Vert_{\dot{H}^{1/2}}
+2 \frac{|c^{\prime}(t)|}{c(t)} \mathscr{E}(w_k(t)). 
\nonumber
\end{align}
Here, we see from \eqref{EQ:G-interval} in 
Proposition \ref{prop:G-Colombini} that
\begin{align*}
& \| v_k(t) \|_{\dot{H}^{3/2}}\| \pa_t w_k(t) \|_{\dot{H}^{1/2}}\\
\le& M^2 e^{4M^2T^{1-(qs-s)}}
\|((-\Delta)^{3/2}v_0,(-\Delta)^{1/2}v_1)
\|^2_{\gamma^s_{\eta,L^2}\times \gamma^s_{\eta,L^2}}
\end{align*}
for $0\le t\le T$.
Then we integrate \eqref{EQ:small-S-diff} and 
apply Gronwall's lemma to obtain
\begin{multline*}
\mathscr{E}(w_k(t))
\le M^2e^{4M^2m^{-1}_0T^{1-(qs-s)}}
\|((-\Delta)^{3/2}v_0,(-\Delta)^{1/2}v_1)
\|^2_{\gamma^s_{\eta,L^2}\times \gamma^s_{\eta,L^2}}\times \\
\left( \int^t_0
\left\vert c_k(\tau)^2-c(\tau)^2 \right\vert \, d \tau \right)
\exp\left(2\int^t_0 \frac{\vert c^{\prime}(\tau) \vert }{c(\tau)} \, d \tau
\right) 
\end{multline*}
for $t \in [0,T)$, 
which implies that 
\[ \left. \begin{gathered} 
\nabla v_k(t) \to \nabla v(t) \\ 
\pa_t v_k(t) \to \pa_t v(t) 
\end{gathered} \right\} \quad 
\text{in $L^{\infty}_{\mathrm{loc}}([0,T);L^2)$ as $k \to \infty$.}
\]
Hence we get \eqref{EQ:small-S-convergence}, 
which proves the continuity of $\Theta$. \\

We are now in a position to conclude the proof. 
Proposition \ref{prop:Ge-Chiave} and the previous results assure that  
$\Theta$ is continuous from 
$\mathscr{K}$ into itself, provided that the data $(v_0,v_1)$ satisfy
\eqref{EQ:2-Ge-data}--\eqref{EQ:R-data}.
Since $\mathscr{K}$ is the convex and compact subset of 
the Fr\'echet space $L^\infty_{\mathrm{loc}}([0,T))$, 
the Schauder-Tychonoff theorem implies that 
$\Theta$ has a fixed point in $\mathscr{K}$, 
and hence, we conclude that solution $v(t,x)$ to 
the linear Cauchy problem 
\eqref{EQ:G-V-Linear}--\eqref{EQ:G-V-initial} is also a solution 
to the nonlinear Cauchy problem \eqref{EQ:Kirchhoff} with data $(v_0,v_1)$ 
on $[0,T]$.

In conclusion, we obtain that
there exist $M>2$ and $R>0$ such that for every $\eta_0(M,R,T)>0$ 
there exists $\eta>\eta_0(M,R,T)$ so that
if $s>1$ and
$(u_0,u_1)\in \gamma_{\eta,L^2}^s\times \gamma_{\eta,L^2}^s$ satisfy
\eqref{EQ:2-Ge-data}--\eqref{EQ:R-data}, 
the Cauchy problem \eqref{EQ:Kirchhoff} admits a solution
$u$ in the class 
$C^1([0,\infty);\gamma_{L^2}^s)$. 
The uniqueness is proved by the same argument 
as in the fixed point one.  
The proof of Theorem \ref{thm:Gevrey thm} is finished. 
\end{proof}


\section{Initial-boundary value problems for the Kirchhoff equation}
\label{sec:4}
The argument in the proof of Theorem \ref{thm:Gevrey thm} is 
available for 
the initial-boundary value problems in  an open set
$\Omega$ of $\Rn$. 
In this section we discuss the global well-posedness for 
initial-boundary value problem to the Kirchhoff equation 
in the typical domains: bounded domains 
and exterior domains. The results in  this section can be proved by 
Fourier series expansions method in bounded domains, and generalised Fourier 
transform method in exterior domains, respectively. 
It is known from spectral theorem that 
a self-adjoint operator on a separable Hilbert space is unitary 
equivalent to a multiplication operator on some $L^2(\mathcal{M},\mu)$,
where $(\mathcal{M},\mu)$ is a measure space. 
Then $L^2(\Omega)$ is unitary equivalent to 
$L^2(\Rn)$. This means that the Fourier transform method in $\Rn$ 
is available for $L^2$ space on an open set $\Omega$ in $\Rn$;
any multiplier acting on $L^2(\Rn)$ is unitarily transformed into an multiplier 
acting on $L^2(\Omega)$.  

\subsection{The case: $\Omega$ is an exterior domain.}
Replacing the Fourier transform over $\Rn$ by the generalised Fourier transform 
over exterior domains and applying exactly the same argument of 
Theorem \ref{thm:Gevrey thm}, 
we can also prove 
a similar result on the initial-boundary value problem in exterior domains. More precisely, we consider the following problem:
\begin{equation}\label{EQ:B-Kirchhoff}
\left\{
\begin{aligned}
& \pa^2_t u-\left(1+\int_\Omega |\nabla u(t,y)|^2 \, dy\right)\Delta u=0, 
& \quad t>0, \quad x\in \Omega,\\
& u(0,x)=u_0(x), \quad \pa_t u(0,x)=u_1(x), &\quad x \in \Omega,\\
& u(t,x)=0, & \quad x\in \pa\Omega.
\end{aligned}\right.
\end{equation} 
Here, $\Omega$ is a domain in $\Rn$ such that 
$\Rn\setminus \Omega$
is compact and its boundary $\pa\Omega$ 
is analytic.  The latter assumption may be in principle relaxed but
this would require an extension of known analytic solvability results to the
Gevrey setting, so we omit it for this moment, and refer to
\cite{Arosio-analytic} and \cite{KoNa} for further details.

Following Wilcox \cite{Wilcox}, let us define 
the generalised Fourier transforms in an arbitrary exterior domain $\Omega$. 
Let $A$ be a self-adjoint realisation of the Dirichlet Laplacian $-\Delta$ 
with domain $H^2(\Omega)\cap H^1_0(\Omega)$. Then $A$ is 
non-negative on $L^2(\Omega)$, and we can define the square root
$A^{1/2}$ of  $A$.
We recall the resolvent operator $R(|\xi|^2\pm i0)$:
$$R(|\xi|^2\pm i0)=
\lim_{\varepsilon\to+0}(A-(|\xi|^2\pm i\varepsilon))^{-1},
$$
and $R(|\xi|^2\pm i0)$ is bounded from 
$L^2(\Omega,\langle x \rangle^s dx)$ to 
$H^2(\Omega,\langle x \rangle^{-s}dx)$ for each $\xi\in \Rn$ and 
some $s>1/2$,
where $\langle x \rangle=(1+|x|)^{1/2}$
(see, e.g., Mochizuki \cite{Mochizuki}).
Introducing a function $j=j(x)\in C^\infty(\Rn)$ vanishing 
in a neighbourhood of $\Rn\setminus \Omega$ and equal to 
one for large $|x|$, let us define the generalised Fourier transforms 
as follows: 
\[ 
(\mathscr{F}_\pm f)(\xi)=\lim_{L\to\infty}
(2\pi)^{-n/2}
\int_{\Omega\cap \{|x|<L\}}
\overline{\psi_\pm(x,\xi)}f(x)\, dx \quad \text{in \quad 
$L^2(\Rn)$,}
\]
where we put 
\[
\psi_\pm(x,\xi)=j(x)e^{ix\cdot\xi}+[R(|\xi|^2\pm i0)M_\xi(\cdot)](x) 
\]
\[
\text{with} \quad M_\xi(x)=-(A-|\xi|^2)(j(x)e^{ix\cdot \xi}).
\]
Notice that we can write formally 
\[
M_\xi(x)=\{\Delta j(x)+2i\xi\cdot \nabla j(x)\}e^{ix\cdot \xi}.
\]
The kernels $\psi_\pm(x,\xi)$ are called eigenfunctions of the operator $A$ 
with eigenvalue $|\xi|^2$ 
in the sense that, formally, 
$$(A-|\xi|^2)\psi_\pm(x,\xi)=0,
$$
but $\psi_\pm(x,\xi)\notin L^2(\Omega)$. Similarly, 
the inverse transforms are defined by 
\[
(\mathscr{F}_\pm^*g)(x)=\lim_{L\to\infty}(2\pi)^{-n/2}\int_{\{|\xi|<L\}}
\psi_\pm(x,\xi)g(\xi)\, d\xi \quad \text{in $L^2(\Omega)$.}
\]
We treat $\mathscr{F}_+f$ only and drop the subscript $+$, since 
$\mathscr{F}_-f$ can be dealt with by essentially the same method. 
The transform $\mathscr{F} f$ thus defined obeys the following properties 
(see, e.g., Shenk II \cite[Theorem 1 and Corollary~5.1]{Shenk}): 
\begin{enumerate}
\item[(i)] $\mathscr{F}$ is a unitary mapping 
$$\mathscr{F}:L^2(\Omega)\rightarrow L^2(\Rn).$$ 
Hence 
$$
\mathscr{F}\mathscr{F}^*=I. 
$$
\item[(ii)] $\mathscr{F}$ satisfies the generalised Parseval identity:
\[
(\mathscr{F}f,\mathscr{F}g)_{L^2(\Rn)}=(f,g)_{L^2(\Omega)}, \quad 
f,g \in L^2(\Omega).
\]

\item[(iii)] $\mathscr{F}$ diagonalises the operator $A$ in the sense that 
\[
\mathscr{F}(\varphi(A)f)(\xi)=\varphi(|\xi|^2)(\mathscr{F}f)(\xi),
\]
where $\varphi(A)$ is the operator 
defined by the spectral representation theorem 
for self-adjoint operators. 
\end{enumerate}

We define the Sobolev spaces over $\Omega$ as follows: 
\[
\text{$H^\sigma(\Omega)=(1+A)^{-\sigma/2}L^2(\Omega)$}
\quad \text{for $\sigma\in \R$,}
\] 
and their homogeneous version is defined as 
\[
\text{$\dot{H}^\sigma(\Omega)=A^{-\sigma/2}L^2(\Omega)$}
\quad \text{for $\sigma\in \R$.}
\]
We say that $f\in \gamma_{L^2}^s(\Omega)$ for $s\ge1$ if 
and only if there exists a constant $\eta>0$ such that
\[
\int_\Rn e^{\eta|\xi|^{1/s}}|(\mathscr{F}f)(\xi)|^2\, d\xi<\infty.
\]
The class $\gamma^s_{L^2}(\Omega)$ is endowed with the 
projective limit topology.\\

To state the results, we need to introduce 
the analytic compatibility condition. \\

\noindent 
{\bf The Gevrey 
compatibility condition.} $f$ 
satisfies the Gevrey 
compatibility condition if and only if 
$f\in \gamma^s_{L^2}(\Omega)$ satisfies 
\[
A^kf\in H^1_0(\Omega), \quad k=0,1,\cdots.
\]

Based on the properties (i)--(iii) of the generalised Fourier transform, 
we have:
\begin{thm} \label{thm:Exterior}
Assume that $\Omega$ is an exterior domain of $\Rn$ such that $\Rn\setminus 
\Omega$ is compact with analytic boundary 
$\pa \Omega$.
Let $T>0$ and $s>1$. 
Suppose that there exist $M>2$, $R>0$ 
and $\eta>\eta_0(M,R,T)$ for a sufficiently large 
$\eta_0(M,R,T)$ such that 
the functions
$u_0,u_1\in\gamma^s_{L^{2}}(\Omega)$
satisfy 
\[
2\mathscr{H}(u;0)<\frac{M^2}{4}-1,
\]
\[
\left\|((-\Delta)^{3/4}u_0,(-\Delta)^{1/4}u_1)
\right\|^2_{\gamma^s_{\eta,L^2}\times \gamma^s_{\eta,L^2}}
\le R.
\]
Then the initial-boundary-value problem \eqref{EQ:B-Kirchhoff} 
admits a unique solution
$u$ in the class 
\[
C^1\left([0,\infty);\gamma_{L^2}^s(\Omega)
\right).
\]
\end{thm}
The constants can be made precise, in the same way as in Remark \ref{REM:eta0}.

\subsection{The case: $\Omega$ is a bounded domain}
Replacing Fourier transform by Fourier series expansion and 
applying exactly the same argument of the proof of Theorem 
\ref{thm:Gevrey thm}, 
we can prove a similar result 
for the initial-boundary value problem in $[0,\infty)\times \Omega$, 
where $\Omega$ is a bounded domain in $\R^n$ with analytic boundary $\pa \Omega$. 
Let $\{w_k\}_{k=1}^\infty$ be a complete orthonormal system 
of eigenfunctions of the operator $-\Delta$ whose domain 
is $H^2(\Omega)\cap H^1_0(\Omega)$, and let 
$\lambda_k$ be eigenvalues corresponding to $w_k$. 
Namely, $\{w_k,\lambda_k\}$ satisfy the elliptic equations: 
\[
\left\{
\begin{aligned}
 -\Delta w_k&=\lambda_k w_k &  \quad \text{in $\Omega$,}\\
 w_k&=0 & \quad \quad \text{on $\pa\Omega$.}
\end{aligned}\right.
\] 
Then $(w_k,w_\ell)_{L^2(\Omega)}=\delta_{k\ell}$ and 
\[
0<\lambda_1\le \lambda_2\le \cdots \le\lambda_k \le \cdots 
\quad \text{and $\lambda_k\to \infty$,}
\] 
where $(\phi,\psi)_{L^2(\Omega)}$ stands for the inner product of $\phi$ and $\psi$ 
in $L^2(\Omega)$. We say that $f\in H^\sigma(\Omega)$ for real $\sigma$ if 
\[
\sum_{k=1}^\infty \lambda^{2\sigma}_k
\left|(f,w_k)_{L^2(\Omega)} \right|^2<\infty,
\]
and $f\in \gamma_{L^2}^s(\Omega)$ for $s\ge1$ if 
and only if there exists a constant $\eta>0$ such that
\[
\sum_{k=1}^\infty e^{\eta\lambda^{1/s}_k}
\left|(f,w_k)_{L^2(\Omega)} \right|^2<\infty.
\]
The class $\gamma_{L^2}^s(\Omega)$ is endowed with the inductive limit topology. \\

Then we have:

\begin{thm} \label{thm:A}
Assume that $\Omega$ is a bounded domain in $\Rn$ with 
analytic boundary $\pa \Omega$.
Let $T>0$ and $s>1$. 
Suppose that there exist $M>2$, $R>0$ 
and $\eta>\eta_0(M,R,T)$ for a sufficiently large 
$\eta_0(M,R,T)$ such that 
the functions
$u_0,u_1\in\gamma^s_{L^{2}}(\Omega)$
satisfy 
\[
2\mathscr{H}(u;0)<\frac{M^2}{4}-1,
\]
\[
\left\|((-\Delta)^{3/4}u_0,(-\Delta)^{1/4}u_1)
\right\|^2_{\gamma^s_{\eta,L^2}\times \gamma^s_{\eta,L^2}}
\le R.
\]
Then the initial-boundary-value problem \eqref{EQ:B-Kirchhoff} 
admits a unique solution
$u$ in the class 
\[
C^1\left([0,\infty);\gamma_{L^2}^s(\Omega)\right).
\]
\end{thm}
Again, the constants can be made precise, in the same way as in Remark \ref{REM:eta0}.


\end{document}